\documentclass{amsart}
\usepackage{amsmath,amssymb}
\usepackage{cite}

\usepackage{graphicx}
\graphicspath{ {pdfs/} }
\usepackage{caption}
\usepackage{subcaption}

\usepackage{hyperref}

\usepackage{todonotes}
\usepackage[T1]{fontenc}
\usepackage[utf8]{inputenc}
\usepackage[italian,english]{babel}

\newtheorem{theorem}{Theorem}[section]
\newtheorem*{theorem*}{Theorem}
\newtheorem{lemma}[theorem]{Lemma}

\newtheorem{proposition}[theorem]{Proposition}
\newtheorem*{question*}{Question}

\theoremstyle{definition}
\newtheorem{definition}[theorem]{Definition}
\newtheorem{remark}[theorem]{Remark}
\newtheorem{example}[theorem]{Example}
\newtheorem{algorithm}[theorem]{Algorithm}
\newtheorem{question}[theorem]{Question}
\newtheorem{conjecture}[theorem]{Conjecture}

\newcommand{\newabstract}[1]{%
  \par\bigskip
  \item[\hskip\labelsep\scshape R\'esum\'e.]
}

\begin{document}

\title[Trisection diagrams and twists of 4-manifolds]{Trisection diagrams and twists of 4-manifolds}

\author[P. Naylor]{Patrick Naylor}
\address{Department of Mathematics, Princeton University, Princeton NJ 08544}
\email{patrick.naylor@princeton.edu}
\urladdr{https://patricknaylor.org}


\begin{abstract}
 A theorem of Katanaga, Saeki, Teragaito, and Yamada relates Gluck and Price twists of 4-manifolds. Using trisection diagrams, we give a purely diagrammatic proof of this theorem, and answer a question of Kim and Miller.

\newabstract{french}
Un th\'eor\`eme de Katanaga, Saeki, Teragaito, and Yamada \'etablit une connexion entre des torsions de Gluck et Price. On donne une nouvelle d\'emonstration de ce th\'eor\`eme en utilisant des diagrammes de trisection, et répond à une question de Kim et Miller.
\end{abstract}

\maketitle

\section{Introduction}\label{sec:intro}

Trisections of 4-manifolds were introduced by Gay and Kirby in 2012 as a 4-dimensional analogue of Heegaard splittings. Recently, they have been used to give new proofs of classical results \cite{Lam18} \cite{Lam19} and understand embedded surfaces in 4-manifolds \cite{MeiZup18}\cite{GayMei1806}. One main feature of trisections is that they may be represented diagrammatically, and thus offer a new perspective with which to view smooth 4-manifolds. As an example of the power of these diagrams, we give a completely new proof of a non-trivial surgery theorem using a purely diagrammatic argument.

Suppose that $S$ is an embedded 2-sphere in a 4-manifold $X$ with regular neighbourhood $N\subset X$ diffeomorphic to $S^2\times D^2$. Let $r_\theta:S^2\to S^2$ be the diffeomorphism which rotates $S^2$ by $\theta$. Originally defined by Gluck \cite{Gluck61}, the \emph{Gluck twist} of $X$ along $S$ is the 4-manifold
\[\Sigma_S(X):=(X-\textnormal{int}(N))\cup_\tau N\]

\noindent where $\tau$ is the diffeomorphism of $S^2\times S^1$ given by $\tau(x,\theta)=(r_\theta(x),\theta)$. This construction is particularly interesting when $X$ is the 4-sphere; $\Sigma_{S^4}(S)$ is a homotopy 4-sphere, and is therefore homeomorphic to $S^4$ by work of Freedman \cite{Freedman82}. However, it remains an open question whether all Gluck twists on $S^4$ are standard, i.e. diffeomorphic to $S^4$.

A similar surgery can be performed along an embedded projective plane. Suppose that $P$ is a projective plane in a 4-manifold $X$ with Euler number $\pm 2$. A regular neighbourhood of $P$ is diffeomorphic to $N_\pm$, a disk bundle over $P$ whose boundary $Q$ is the quaternion space (the quotient of $S^3$ by the action of the quaternion group). If $\phi$ is a self-diffeomorphism of $Q$, the 4-manifold
\[\Pi_{P,\phi}(X):=(X-\textnormal{int}(N_\pm))\cup_\phi N_\pm\]

\noindent is called a Price surgery along $P$. Price \cite{Pri77} studied the self-diffeomorphisms of $Q$, and showed that there are only six classes up to isotopy. In particular, up to isotopy there is only one non-trivial self-diffeomorphism of $Q$ which could be used to produce a 4-manifold homeomorphic  (but perhaps not diffeomorphic) to $X$. Consequently, the resulting 4-manifold is called the \emph{Price twist} of $X$ along $P$, and we will denote it by $\Pi_P(X)$. Like the Gluck twist, this surgery is known to produce exotic 4-manifolds in some settings \cite{Akb09}, but is most interesting in the case that $X$ is the 4-sphere. Note that by a theorem of Massey \cite{Mas69}, all embedded projective planes in $S^4$ have Euler number $\pm 2$.

In this paper, we use trisection diagrams to give an entirely new proof of the following theorem that relates these surgeries, proved by Katanaga, Saeki, Teragaito, and Yamada \cite{KatSaeTerYam99}. This is made possible by recent work on trisection diagrams of complements of surfaces in 4-manifolds; the existence of a purely trisection-diagrammatic proof of this theorem answers a question of Kim and Miller \cite{KimMil1808}.

\begin{theorem}[\hspace{1sp}\cite{KatSaeTerYam99}]\label{kstythm}
	Let $X$ be a 4-manifold. Let $S\subset X$ be an embedded sphere with Euler number 0, and let $P_\pm\subset X$ be an unknotted projective plane with Euler number $\pm2$. Then $\Sigma_S(X)$ is diffeomorphic to $\Pi_{S\#P_{\pm}}(X)$.
\end{theorem}

Trisection diagrams are very similar to Heegaard diagrams, but with three sets of curves. A diagram encodes a smooth closed 4-manifold, and after a suitable stabilization operation (as in the Reidemiester-Singer theorem for Heegaard splittings), any two diagrams for the same 4-manifold are related by a surface automorphism, and isotopy and slides of curves of each type. After carefully setting up trisection diagrams for $\Sigma_S(X)$ and $\Pi_{S\#P_{\pm}}(X)$, the proof is a step-by-step verification that these diagrams are related by allowable moves. A priori, one might expect that arbitrary stabilizations might be needed in the proof, but surprisingly this is \emph{not} the case. Consequently, the calculation in this paper provides evidence that trisection diagrams are an effective computational tool for working with smooth 4-manifolds.

\subsection*{Organization}

This paper is organized as follows. In Section \ref{sec:trisections}, we review trisections and trisection diagrams. In Section \ref{sec:surfacecomplements}, we briefly review work of Gay-Meier and Kim-Miller on trisection diagrams of complements of surfaces in 4-manifolds, and build the requisite diagrams. Finally, in Section \ref{sec:proof} we give a purely trisection-diagrammatic proof of Theorem \ref{kstythm}.

\subsection*{Acknowledgements}

This work was supported by NSERC CGS-D and CGS-MSFSS scholarships. Much of this work was done on a visit to the University of Georgia, and the author would like to thank Sarah Blackwell, David Gay, Jason Joseph, Jeffrey Meier, William Olsen, and Adam Saltz for their hospitality and many encouraging conversations, as well as his graduate advisor, Doug Park. The author would also like to thank an anonymous referee for reading this paper and providing many helpful comments.

\section{Trisections of 4-manifolds}\label{sec:trisections}

\subsection{Trisections and trisection diagrams}

In this section we briefly review the definition of a trisection and a trisection diagram. For more exposition the reader is referred to \cite{GayKir16}, \cite{CasGayPin18b}, and \cite{MeiZup18}.

\begin{definition}
A \emph{handlebody} of genus $g$ is a compact, orientable manifold which can be built with a single 0-handle and $g$ 1-handles.
\end{definition}

\begin{definition}[\hspace{1sp}\cite{GayKir16}]
Suppose that $X$ is a smooth, oriented, closed, and connected 4-manifold. A trisection $\mathcal{T}$ of $X$ is a decomposition $X=X_1\cup X_2\cup X_3$ such that:
	\begin{itemize}
		\item $X_i$ is diffeomorphic to a 4-dimensional handlebody of genus $k_i$;
		\item $X_i\cap X_j$ is diffeomorphic to a 3-dimensional handlebody of genus $g$;
		\item $X_1\cap X_2\cap X_3\cong \Sigma_g$, a closed surface of genus $g$.
	\end{itemize}
\end{definition}

We will refer to this as a $(g;k_1,k_2,k_3)$-trisection of $X$. When $k_1=k_2=k_3$ the trisection is called \emph{balanced}, and we refer to it as a $(g;k)$-trisection. Each $X_i$ is called a \emph{sector}, and the triple intersection is called the \emph{central surface} of $\mathcal{T}$. Note that the central surface induces a genus $g$ Heegaard splitting of $\partial X_i\cong \#^{k_i}S^1\times S^2$.

\begin{example}
The simplest trisection is of $S^4$. If $S^4\subset \mathbb{R}^5=\mathbb{C}\times \mathbb{R}^3$ is parametrized as
\[S^4=\{(re^i\theta,x,y,z):|(re^i\theta,x,y,z)|=1\},\]
then we can define three sectors by
\[X_k=\{(re^i\theta,x,y,z)\in S^4: 2\pi k/3\leq \theta\leq 2\pi (k+1)/3\}.\]
It is easy to check that each $X_k$ is a 4-ball, and that in fact $X_1\cap X_2\cap X_3$ is an unknotted 2-sphere (it is the collection of points where $r=0$). Consequently, this is a $(0;0)$-trisection of $S^4$. In fact, any $(0;0)$-trisection is diffeomorphic to this one.
\end{example}

There is a natural stabilization operation for trisections of a fixed 4-manifold.

\begin{definition}\label{StabilizationDefinition}
Suppose that $\mathcal{T}$ is a $(g;k_1,k_2,k_3)$-trisection of a 4-manifold $X$, with sectors $X_1$, $X_2$ and $X_3$. Let $\alpha\subset X_1\cap X_2$ be a properly embedded and boundary parallel arc, and define a new trisection $\mathcal{T}'$ of $X$ by:
\begin{itemize}
\item $X_1'=X_1\setminus \nu(\alpha)$;
\item $X_2'=X_2\setminus \nu(\alpha)$;
\item $X_3'=X_3\cup \nu(\alpha)$.
\end{itemize}
One can check that this decomposition is a $(g+1;k_1,k_2,k_3+1)$-trisection of $X$, and that this operation is well defined up to isotopy of trisections. The trisection $\mathcal{T}'$ is called a \emph{$3$-stabilization} (or simply \emph{stabilization}) of $\mathcal{T}$, and $\mathcal{T}$ is called a \emph{destabilization} of $\mathcal{T}'$. One can define $1$- and $2$- stabilizations analogously.
\end{definition}

The reader may wish to compare Definition \ref{StabilizationDefinition} to the usual stabilization operation for Heegaard splittings. Note that this process stabilizes the Heegaard splittings of $\partial X_1$ and $\partial X_2$, while adding an $S^1\times S^2$ summand to $\partial X_3$ (in the case of 3-stabilization). Similar to the case of Heegaard splittings, one can also view stabilization as the connected sum (respecting the trisection structure) of $\mathcal{T}$ with one of the three possible genus one trisections of $S^4$ obtained by stabilizing the trivial $(0;0)$-trisection of $S^4$.

The following fundamental result allows us to study closed 4-manifolds via trisections:

\begin{theorem*}[\hspace{1sp}\cite{GayKir16}]
Every smooth, oriented, closed, and connected 4-manifold $X$ admits a $(g;k)$-trisection for some $0\leq k\leq g$. Any two trisections of $X$ become isotopic after sufficiently many stabilizations.
\end{theorem*}

A key feature of trisections is that they can also be described diagrammatically. Indeed, by a classical theorem of Laudenbach and Po\'enaru \cite{LauPoe72}, a trisection is determined by its \emph{spine} (the subset $\bigcup (X_i\cap X_j)$). This in turn can be built from the Heegaard splittings of $\partial X_i$, which may be recorded with a diagram.

\begin{definition}
A \emph{$(g;k_1,k_2,k_3)$-trisection diagram} is a tuple $\mathfrak{D}=(\Sigma;\alpha,\beta,\gamma)$, where $\Sigma$ is a closed orientable surface of genus $g$, and $\alpha$, $\beta$, and $\gamma$ are collections of $g$ embedded closed curves such that:
\begin{itemize}
\item Each of $\alpha,\beta$, and $\gamma$ is a \emph{cut system of curves} for $\Sigma$, i.e. surgery on each set of curves yields $S^2$;
\item Each pair of curves is standard, i.e. each of $(\Sigma;\alpha,\beta)$, $(\Sigma;\beta,\gamma)$, and $(\Sigma;\gamma,\alpha)$ is a genus $g$ Heegaard diagram for $\#^{k_i} S^1\times S^2$.
\end{itemize}
\end{definition}

By Waldhausen's theorem, there is a unique Heegaard splitting for $\#^{k_i}S^1\times S^2$, and so any pair of $\alpha$, $\beta$, and $\gamma$ can be standardized by handle slides. However, the three sets of curves are usually not \emph{simultaneously} standard.

A trisection diagram determines a trisected 4-manifold up to diffeomorphism in the following way. Beginning with $\Sigma\times D^2$, attach thickened 3-dimensional handlebodies corresponding to the $\alpha$, $\beta$, and $\gamma$ curves along $\Sigma\times \{0\}$, $\Sigma\times \{2\pi/3\}$, and $\Sigma\times \{4\pi/3\}$ respectively, where $D^2$ is thought of as the unit disk in $\mathbb{C}$. By assumption, the three boundary components of the resulting 4-manifold are diffeomorphic to $\#^{k_i}S^1\times S^2$, and so by a theorem of Laudenbach-Po\'enaru \cite{LauPoe72} they can be uniquely filled in (up to diffeomorphism) with $\natural^{k_i}S^1\times B^3$ to obtain a closed trisected 4-manifold. 

The simplest trisection diagram encodes the $(0;0)$-trisection of $S^4$ described above. It consists of a 2-sphere, with \emph{no} curves. The reader may wish to follow the construction given above to see this is the case.  In particular, the following diagrams describe the three possible stabilizations of the trivial $(0;0)$-trisection of $S^4$. Note that exactly one sector in each trisection is diffeomorphic to $S^1\times B^3$, and the other two are diffeomorphic to $B^4$.

\begin{figure}[ht]
\centering
\includegraphics[height=.15\linewidth]{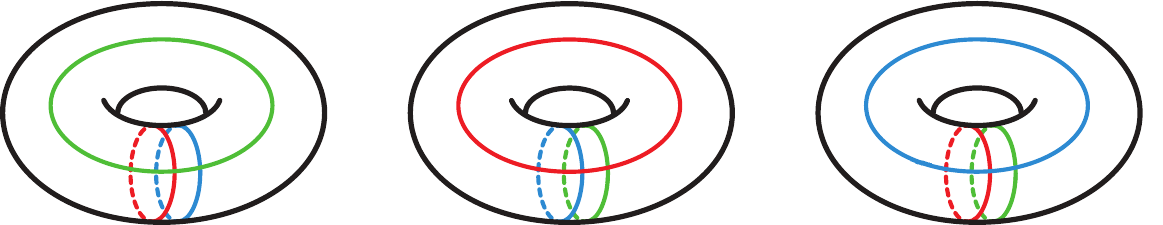}
\caption{The three genus one (unbalanced) trisection diagrams for $S^4$, obtained by stabilizing the $(0;0)$-trisection of $S^4$.}
\label{TrisectionDiagramStab}
\end{figure}

Stabilizing a trisection may also be represented diagrammatically. In general, if $\mathfrak{D}_1$ and $\mathfrak{D}_2$ are trisection  diagrams for $X_1$ and $X_2$, then $\mathfrak{D}_1\#\mathfrak{D}_2$ is a trisection diagram for the natural trisection of $X_1\# X_2$ obtained by performing the connected sum at points on the central surfaces. Note that up to handle slides and isotopy of the curves, it does not matter how this connected sum of diagrams is performed. In particular, by the remark following Definition \ref{StabilizationDefinition}, stabilization can be thought of as a connected sum with a genus one trisection of $S^4$. Consequently, we give the following definition. For more exposition, the reader is referred to \cite{MeiSchZup16}.

\begin{definition}
Suppose that $(\Sigma;\alpha,\beta,\gamma)$ is a trisection diagram for $X$. If $S$ is one of the genus one trisection diagrams for $S^4$ in Figure \ref{TrisectionDiagramStab}, then $\Sigma\# S$ is also a trisection diagram for $X$, and we call $\Sigma\# S$ a \emph{stabilization} of $\Sigma$. Conversely, $\Sigma$ is called a \emph{destabilization} of $\Sigma\# S$.
\end{definition}

Besides stabilization, there are other moves on trisection diagrams that do not change the resulting 4-manifold. In particular, isotopy of the curves in a diagram, or applying a global surface automorphism obviously do not change the resulting 4-manifold. The following theorem allows us to understand smooth closed 4-manifolds via their trisection diagrams.

\begin{theorem*}[\hspace{1sp}\cite{GayKir16}]
Every trisection of a 4-manifold can be represented by a trisection diagram. Moreover, two trisection diagrams describe diffeomorphic 4-manifolds if and only if they are related by stabilization, handle slides and isotopy of curves (among curves of the same type), and surface diffeomorphisms.
\end{theorem*}

The reader may wish to compare this theorem with the analogous statement for Heegaard splittings. Recall that a handle slide of a curve $\alpha_1$ over $\alpha_2$ in $\Sigma$ is simply a third curve $\alpha_3$ with the property that $\alpha_1$, $\alpha_2$, and $\alpha_3$ bound an embedded pair of pants $P\subset \Sigma$. 

In general, it may not be obvious whether two trisection diagrams describe diffeomorphic 4-manifolds. If they do, arbitrarily many stabilizations might be required to relate them by handle slides. It is also usually difficult to decide if a given trisection diagram can be destabilized, since in principle, one must rearrange the curves to realize the diagram as a connected sum with one of the stabilizations in Figure \ref{TrisectionDiagramStab}. Alternatively, Lemma \ref{lem:destab} gives a slightly more practical condition that can be used to recognize a destabilization, and we will used it frequently in Section \ref{sec:proof}.

\begin{lemma}\label{lem:destab}
Suppose that $\mathfrak{D}=(\Sigma;\alpha,\beta,\gamma)$ is a trisection diagram, and that $\alpha_0\in \alpha$, $\beta_0\in \beta$, and $\gamma_0\in \gamma$ are three curves with the property that:
\begin{itemize}
\item Two of $\alpha_0$, $\beta_0$, and $\gamma_0$ are parallel;
\item The remaining curve intersects these parallel curves each exactly once.
\end{itemize}
Then $\mathfrak{D}$ can be destabilized. To do so, we can simply erase $\alpha_0$, $\beta_0$, and $\gamma_0$ from $\Sigma$ and surger $\Sigma$ along \textnormal{any} of these curves.
\end{lemma}

\begin{proof}
By hypothesis, the diagram $\mathfrak{D}$ decomposes as a connected sum $\mathfrak{D}'\#S$, where $S$ is one of the stabilized diagrams in Figure \ref{TrisectionDiagramStab}. Since destabilization is equivalent to surgering any of the curves in $S$, this completes the proof.
\end{proof}

Trisection diagrams can be quite complicated in general, but some standard 4-manifolds admit diagrams of low genus. Some examples are given below.

\begin{example}
	Figures \ref{TrisectionDiagramCP2} and \ref{TrisectionDiagramS2XS2} illustrate minimal genus diagrams of some well known simply connected 4-manifolds. Using the algorithm outlined in \cite{GayKir16}, one can convert these trisection diagrams into Kirby diagrams to verify that they describe the correct 4-manifolds.

\begin{figure}[ht]
\centering
\begin{minipage}{.5\textwidth}
  \centering
  \includegraphics[height=.3\linewidth]{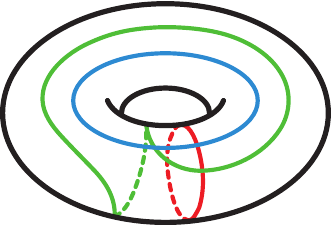}
  \captionsetup{width=.8\linewidth}
  \captionof{figure}{A $(1;0)$-trisection diagram for $\mathbb{C}\mathbb{P}^2$.}
	\label{TrisectionDiagramCP2}
\end{minipage}%
\begin{minipage}{.5\textwidth}
  \centering
  \includegraphics[height=.3\linewidth]{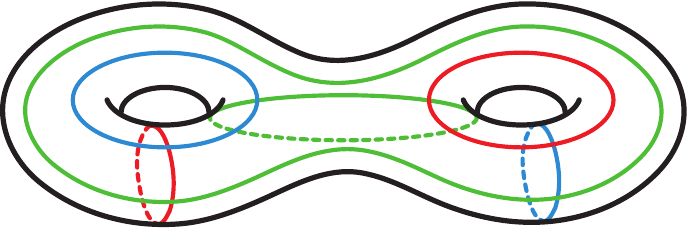}
  \captionsetup{width=.8\linewidth}
  \captionof{figure}{A $(2;0)$-trisection diagram for $S^2\times S^2$.}
	\label{TrisectionDiagramS2XS2}
\end{minipage}
\end{figure}

\end{example}

\subsection{Trisections of 4-manifolds with boundary}

Trisecting 4-manifolds with boundary is more technical than the closed case. A relative trisection of a 4-manifold $X$ with boundary also decomposes $\partial X$ into three pieces, and in particular induces an open book decomposition on $\partial X$. In this paper, we will only consider cases where $\partial X$ is connected.

\begin{definition}
Suppose that $\Sigma$ is an orientable, connected surface with non-empty boundary. A \emph{compression body} on $\Sigma$ is a 3-manifold $C$ obtained by attaching 3-dimensional 2-handles to a thickening of $\Sigma$, i.e.
\[C=\Sigma\times [0,1]\cup_{\Sigma\times \{1\}} \{\text{3-dimensional 2-handles}\}\]

\noindent The boundary of $C$ decomposes as $\partial C=(\partial_-\Sigma)\cup (\partial \Sigma\times [0,1])\cup(\partial_+\Sigma)$, where
\[\partial_-C=\Sigma\times \{0\},\]

\noindent and
\[\partial_+C=\partial \Sigma\setminus (\partial_-\Sigma\cup\partial\Sigma\times (0,1)).\]

\noindent We will often assume that $\partial_+C$ is connected.
\end{definition}

We will now describe specific decompositions of a 4-dimensional 1-handlebody; like the closed case, these will make up the sectors of a relative trisection.

\begin{definition}
Let $\Sigma$ be an orientable and connected surface with non-empty boundary, and let $C$ be a compression body on $\Sigma$. Note that $Z=C\times [0,1]$ is a 4-dimensional 1-handlebody. Consider the decomposition $\partial Z=\partial_{\textnormal{\hspace{0.1em}in}}Z\cup \partial_{\textnormal{out}}Z$, where
\[\partial_{\textnormal{\hspace{0.1em}in}}Z=(C\times \{0\})\cup (\partial_-C\times [0,1])\cup(C\times \{1\}),\]

\noindent and
\[\partial_{\textnormal{out}}Z=(\partial\Sigma\times [0,1]\times [0,1])\cup(\partial_+C\times [0,1]).\]

\noindent The portion $\partial_{\textnormal{\hspace{0.1em}in}}Z$ admits a (generalized) Heegaard splitting as $\partial_{\textnormal{\hspace{0.1em}in}}Z=Y^-\cup Y^+$, where
\[Y^-=(C\times \{0\})\cup (\partial_-C\times [0,1/2])\]

\noindent and
\[Y^+=(\partial_-C\times [1/2,1])\cup(C\times \{1\}).\]

\noindent In particular, the splitting surface is $Y^-\cap Y^+=\partial_-C\times \{1/2\}$. Any Heegaard splitting of $\partial_{\textnormal{\hspace{0.1em}in}}Z$ obtained from this one by stabilization is called \emph{standard}. For brevity, we will continue to denote any such decomposition of $\partial_{\textnormal{\hspace{0.1em}in}}Z$ by $(Y^+,Y^-)$.
\end{definition}

With these models in mind, we can now define a relative trisection.

\begin{definition}
	Let $X$ be a smooth, oriented, and connected 4-manifold with connected non-empty boundary. A relative trisection $\mathcal{T}$ of $X$ is a a decomposition $X=X_1\cup X_2 \cup X_3$ such that:
  \begin{itemize}
    \item There are diffeomorphisms $\phi_i:X_i\to Z$ such that $\phi_i(X_i\cap \partial X)=\partial_{\textnormal{out}}Z$,
    \item For each $i$, $\phi_i(X_i\cap X_{i-1})=Y^-$ and $\phi_i(X_i\cap X_{i+1})=Y^+$.
  \end{itemize}
\end{definition}

The advantage of this structure is that it naturally induces an open book on $\partial X$ with binding $L=\partial (X_1\cap X_2\cap X_3)$, for which the surfaces $X_i\cap X_j\cap \partial X$ are pages. Indeed, by construction $\partial X\setminus \nu(L)$ fibers over $S^1$, with fiber diffeomorphic to $\partial_+C$. In particular, the binding is a $|\partial \Sigma|$-component link in $\partial X$, and the pages have genus $g(\partial_+C)$. For more details, the reader is encouraged to consult \cite{CasGayPin18a} and \cite{GayKir16}. If $\partial X$ is connected, then the pages of this open book decomposition are also necessarily connected; this key observation is required to compute relative trisections of surface complements in \cite{KimMil1808}.

Analogous to the closed case, the following fundamental result allows us to study 4-manifolds with boundary via relative trisections.

\begin{theorem*}[\hspace{1sp}\cite{GayKir16}]
Let $X$ be a smooth, oriented, and connected 4-manifold with connected non-empty boundary, and fix an open book decomposition of $\partial X$. Then there is a relative trisection of $X$ inducing this open book decomposition. Any two relative trisections for $X$ inducing isotopic open books on $\partial X$ become isotopic after sufficiently many interior stabilizations.
\end{theorem*}

There are also moves relating relative trisections inducing \emph{different} open book decompositions, but we will not discuss them in this paper. See \cite{RelLInv} and \cite{Cas15} for more details. A key feature of relative trisections is that two such decompositions can be glued together to form a closed (trisected) 4-manifold. The following gluing theorem was originally proved by Castro in his thesis \cite{Cas15}.

\begin{theorem*}[\hspace{1sp}\cite{Cas15}]
Let $\mathcal{T}$ and $\mathcal{T}'$ be trisections of 4-manifolds $X$ and $X'$, respectively. Denote the open book decompositions induced on $\partial X$ and $\partial X'$ by $\mathcal{O}$ and $\mathcal{O}'$, respectively. Suppose that there is a diffeomorphism $f:\partial X\to \partial X'$, and that $f(\mathcal{O})$ is isotopic to $\mathcal{O}'$. Then there is a naturally induced trisection $\mathcal{T}\cup \mathcal{T}'$ of $X\cup_f X'$.
\end{theorem*}

The main advantage of using such specific decompositions of $\partial_{in}Z$ is that one can define relative trisection diagrams.

\begin{definition}
A $(g,k;p,b)$-relative trisection diagram is a tuple $(\Sigma;\alpha,\beta,\gamma)$, where $\Sigma$ is a connected surface with non-empty boundary, $\alpha$, $\beta$, $\gamma$ are collections of disjoint embedded curves, and $(\Sigma;\alpha,\beta)$, $(\Sigma;\beta,\gamma)$, $(\Sigma;\gamma,\alpha)$ are \emph{slide-standard}, i.e. diffeomorphic to the diagram in Figure \ref{fig:relstandard} after handle slides.
\end{definition}

\begin{figure}[ht]
\centering
\includegraphics[width=.9\linewidth]{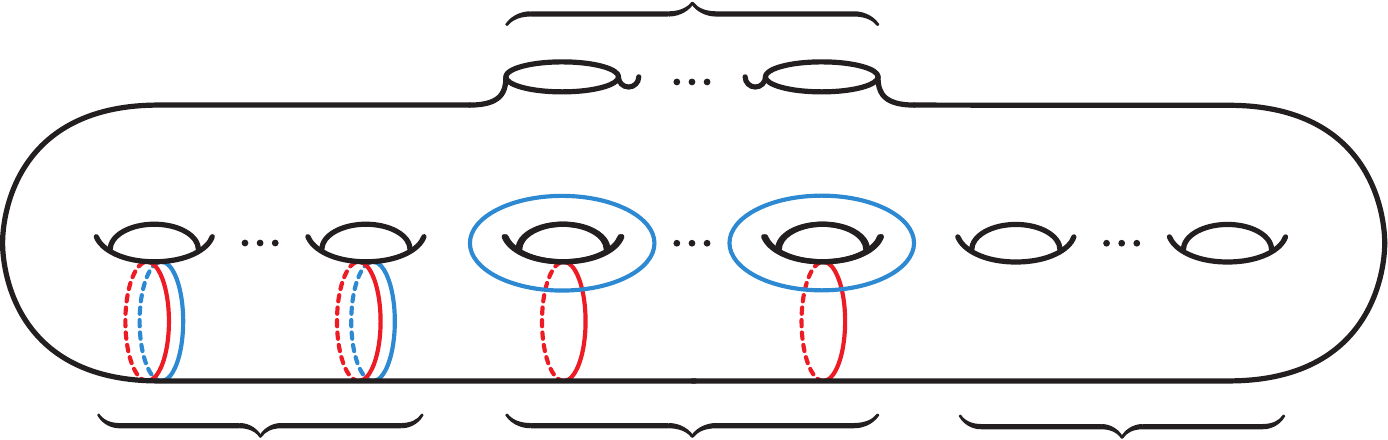}
\put(-294,-12){$k-2p-b+1$}
\put(-199,-12){$g+p+b-k-1$}
\put(-64,-12){$p$}
\put(-223,110){$b>0$ boundary components}
\caption{A standard set of curves for a $(g,k;p,b)$-trisection diagram. The surface has genus $g$ and $b>0$ boundary components, and the result of compressing either set of curves has genus $p$. The integer $k$ records the genus of the 4-dimensional handlebodies of the relative trisection.}
\label{fig:relstandard}
\end{figure}

As in the closed case, a relative trisection diagram $\mathfrak{D}$ determines a trisected 4-manifold, and up to stabilization, two relative trisection diagrams inducing the same open book decomposition are related by a sequence of isotopies and handle slides of curves, and surface automorphisms. In \cite{CasGayPin18a}, the authors also show how to compute the abstract monodromy of the open book decomposition induced by a relative trisection diagram. We will summarize the algorithm here, but refer the reader to \cite{CasGayPin18a} for more details. This algorithm begins by standardizing the $\alpha$ and $\beta$ curves, but this is not strictly necessary; one can state a version of this algorithm which does not require this.

\begin{algorithm}[Monodromy Algorithm]\label{alg:monodromy}
Suppose that $\mathfrak{D}=(\Sigma;\alpha,\beta,\gamma)$ is a relative trisection diagram for $X$. We will denote the result of compressing $\Sigma$ along the $\alpha$ curves by $\Sigma_\alpha$. This is diffeomorphic to a page for the open book decomposition on $\partial X$, and the monodromy will be described as an automorphism of $\Sigma_\alpha$.

\begin{itemize}
\item[(1)] Standardize the $\alpha$ and $\beta$ curves, i.e. perform handle slides so that they look like the curves in Figure \ref{fig:relstandard} (this is often already the case). Let $a$ be a collection of disjoint properly embedded arcs that are disjoint from $\alpha$ and $\beta$, such that the result of compressing $\Sigma\setminus a$ is a disk.
\item[(2)] Do slides of $\beta$ and $\gamma$ curves and slide $a$ over $\beta$ until $a$ is transformed into a collection of arcs, $c$, disjoint from $\gamma$. 
\item[(3)] Let $\alpha'$ be another copy of the $\alpha$ curves. Do slides of the $\gamma$ and $\alpha$ curves and slides of $c$ over $\alpha$, until $c$ is transformed into a new collection of arcs, $a'$, which are disjoint from $\alpha'$.
\item[(4)] Perform slides of $\alpha$ and $a'$ over $\alpha$ until $\alpha'=\alpha$ (in practice, this is often already the case), and $a'$ is another set of arcs disjoint from $\alpha$. The required automorphism $\phi$ of $\Sigma_\alpha$ is now described by $\phi(a)=a'$.
\end{itemize}
\end{algorithm}

There are many choices appearing in this algorithm, but the work of \cite{CasGayPin18a} shows that the monodromy is independent of these choices.

\begin{definition}
Suppose that $\mathfrak{D}=(\Sigma;\alpha,\beta,\gamma)$ is a relative trisection diagram. An \emph{arced relative trisection diagram} is a diagram $(\Sigma;\alpha,\beta,\gamma;a,b,c)$, where $a$ and $c$ are a choice of arcs in $\Sigma$ appearing in Algorithm \ref{alg:monodromy} and $b$ is another copy of $a$.
\end{definition}

\begin{example}
Figure \ref{TrisectionDiagramB4} illustrates a relative trisection diagram for $B^4$. There are two boundary components, and the induced open book decomposition on $S^3$ has annular pages. Using Algorithm \ref{alg:monodromy}, one can compute the induced abstract monodromy of this open book decomposition.

\begin{figure}[ht]
  \centering
  \includegraphics[height=.3\linewidth]{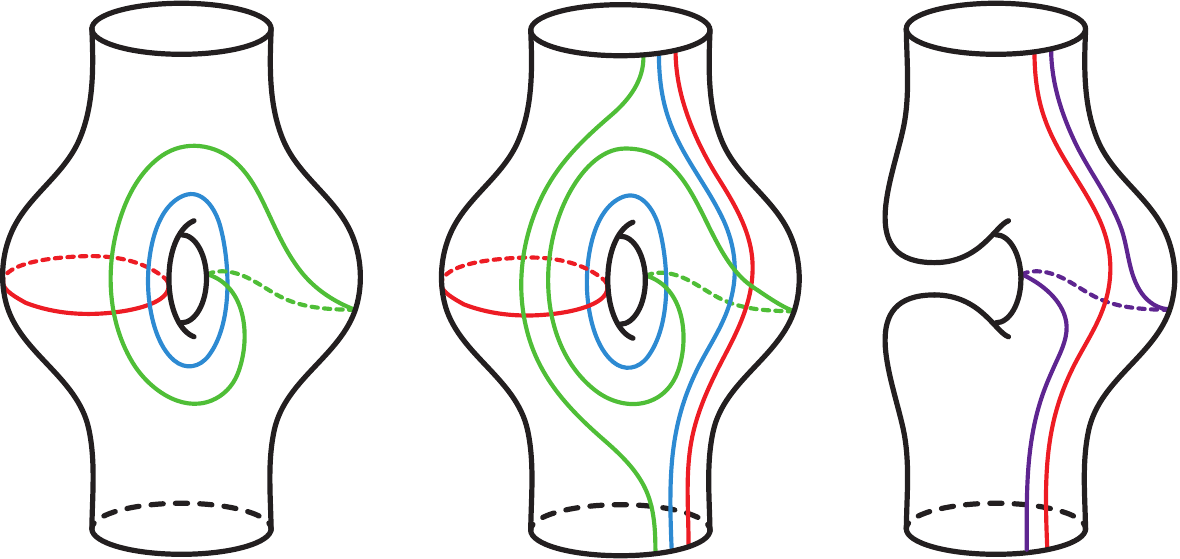}
  \put(-198,-20){(a)}
  \put(-113,-20){(b)}
  \put(-40,-20){(c)}
  \captionof{figure}{In (a), a $(1,1;0,2)$-relative trisection for $B^4$. In (b), an \textit{arced} relative trisection for $B^4$. In (c), the result of applying the monodromy algorithm: an arc and its image in $\Sigma_\alpha$ under the induced monodromy. The open book decomposition induced on $S^3$ has annular pages and monodromy given by a single left handed Dehn twist.}
	\label{TrisectionDiagramB4}
\end{figure}
\end{example}

In combination with the gluing theorem of Castro, the monodromy algorithm can be used to glue relative trisection diagrams. Indeed, suppose that $\mathfrak{D}=(\Sigma;\alpha,\beta,\gamma)$ and $\mathfrak{D}'=(\Sigma';\alpha',\beta',\gamma')$ are two relative trisection diagrams for $X$ and $X'$, and that the induced open books $\mathcal{O}$ and $\mathcal{O}'$ on $\partial X$ and $\partial X'$ are diffeomorphic.
Moreover, assume that $f:\partial X\to \partial X'$ is a diffeomorphism that respects the open books $\mathcal{O}$ and $\mathcal{O}'$.
First, choose a cut system of arcs $a$ for $\Sigma_\alpha$, and use Algorithm \ref{alg:monodromy} to obtain an arced relative trisection diagram $(\Sigma;\alpha,\beta,\gamma;a,b,c)$ for $X$. The image $a'=f(a)$ of $a$ is a cut system of arcs for $\Sigma'_{\alpha'}$, and we can use Algorithm \ref{alg:monodromy} to complete this to an arced relative trisection diagram $(\Sigma';\alpha',\beta',\gamma';a',b',c')$ for $X'$.
Then a relative trisection diagram for $X\cup_fX'$ is given by $(\Sigma\cup_f \Sigma',\alpha'',\beta'',\gamma'')$, where:
\begin{align*}
\alpha''&=\alpha\cup \alpha'\cup (a\cup a') \\
\beta''&=\beta\cup \beta'\cup (b\cup b') \\
\gamma''&=\gamma\cup \gamma'\cup (c\cup c')
\end{align*}

Here, we use the gluing map $f$ to identify the boundary components of $\Sigma$ and $\Sigma'$. In some cases, (when the open book decomposition has annular pages) this process is straightforward. For more details, the reader is encouraged to consult \cite{CasGayPin18a} or \cite{GayMei1806}.

\subsection{Bridge trisections of surfaces}

In \cite{MeiZup17} and \cite{MeiZup18}, Meier and Zupan generalized bridge splittings of knots in $S^3$ to knotted surfaces in 4-manifolds.

\begin{definition}[\hspace{1sp}\cite{MeiZup18}]
	Suppose that a closed 4-manifold $X$ has a $(g;k_1,k_2,k_3)$-trisection $\mathcal{T}$, with sectors $X_1$, $X_2$, and $X_3$. An embedded surface $S\subset X$ is in \emph{bridge position} with respect to $\mathcal{T}$ if:
	\begin{itemize}
		\item $S\cap X_i=\mathcal{D}_i$ is a trivial $c_i$-disk system,
		\item $S\cap (X_i\cap X_j)=\tau_{ij}$ is a trivial $b$-tangle,
    \item $S\cap (X_1\cap X_2\cap X_3)$ is a collection of $2b$ points.
	\end{itemize}

	\noindent Here, a trivial $c_i$-disk system is a collection of $c_i$ properly embedded and boundary parallel disks in $X_i$, and a trivial $b$-tangle is a collection of $b$ properly embedded and boundary parallel arcs in $X_i\cap X_j$. The surface $S$ is said to be in $(b;c_1,c_2,c_3)$-\emph{bridge trisected position} with respect to $\mathcal{T}$. If $c_1=c_2=c_3=c$, the bridge trisection is called balanced, and we will refer to this as a $(g,k;b,c)$-bridge trisection.
\end{definition}

Note that since each $\mathcal{D}_i$ is boundary parallel, the union of any pair of tangles is necessarily an unlink. In fact, the unlink bounds a unique collection of boundary parallel disks in $\natural^kS^1\times B^3$ up to isotopy (rel boundary), and so a bridge trisection is completely determined by the union $\tau_{12}\cup\tau_{23}\cup\tau_{31}$.

In \cite{MeiZup18}, Meier and Zupan show that if $\mathcal{T}$ is a trisection of a 4-manifold $X$ and $S\subset X$ is an embedded surface, then $S$ can be isotoped to lie in bridge trisected position with respect to $\mathcal{T}$. Analogous to the natural stabilization operation for bridge splittings of knots in $S^3$, there is a stabilization operation for bridge trisections with respect to a fixed trisection \cite{MeiZup17}. Hughes, Kim, and Miller \cite{HugSeuMil19} have shown that any two bridge trisections for $S\subset X$ can be made isotopic after some number of stabilizations. We will not need this stabilization operation in this paper, and so refer the reader to \cite{MeiZup17} for more details.

If the 4-manifold in question is $S^4$ together with the $(0;0)$-trisection, then a \emph{tri-plane diagram} is a depiction of each $\tau_{ij}\subset S^3$. Meier and Zupan give a complete calculus of moves that can be used to pass between any two tri-plane diagrams of the same surface in $S^4$. An example of a $(4;2)$-tri-plane diagram is illustrated in Figure \ref{TriplaneDiagramSpunTrefoil}. 

In general, we cannot easily draw diagrams of tangles in $\#^kS^1\times S^2$. Instead, we draw projections of $\tau_{ij}$ on the central surface $\Sigma$ for $\mathcal{T}$ (since $\tau_{ij}$ are boundary parallel). These are called \emph{shadows} for the bridge trisection, which we will denote by $s_{ij}$. Note that for any choice of $s_{ij}$, the union $s_{ij}\cup\tau_{ij}$ bounds a disk in $X_i\cap X_j$. While there may be many different choices of shadows (and corresponding disks in $X_i\cap X_j$), any two choices of shadows for $\tau_{ij}$ are related by \emph{disk slides}.
These may be realized in $\Sigma$ by sliding one shadow over another. A \emph{shadow diagram} of the trivial tangles in Figure \ref{TriplaneDiagramSpunTrefoil} is illustrated in Figure \ref{ShadowDiagramSpunTrefoil}. 

\begin{example}\label{rp2triplane}
	We illustrate a bridge trisection for the spun trefoil $S\subset S^4$. With respect to the trivial trisection of $S^4$, $S$ can be described by the triplane diagram in Figure \ref{TriplaneDiagramSpunTrefoil}. It is also described by the shadow diagram in Figure \ref{ShadowDiagramSpunTrefoil}.

	\begin{figure}[ht]
		\centering
		\includegraphics[width=\linewidth]{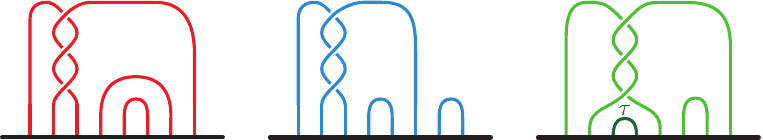}
		\caption{A $(4;2)$-triplane diagram for $S$ consisting of three tangle diagrams in $B^3$. Any tangle together with the mirror of any other tangle gives a 2-component unlink. The arc $\tau$ will be used for meridional stabilization in Figure \ref{TriplaneDiagramSpunTrefoilStabilized}.}%
		\label{TriplaneDiagramSpunTrefoil}
	\end{figure}

	\begin{figure}[ht]
		\centering
		\includegraphics[width=\linewidth]{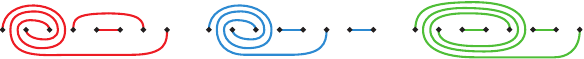}
		\caption{An alternate diagram for the bridge trisection of $S$ in Figure \ref{TriplaneDiagramSpunTrefoil}, consisting of three shadow diagrams in $S^2$. Different choices of shadows for the tangles in Figure \ref{TriplaneDiagramSpunTrefoil} are related by slides of the shadows among each other.}%
		\label{ShadowDiagramSpunTrefoil}
	\end{figure}

\end{example}

It is sometimes desirable to arrange the bridge trisection of $S\subset X$ to have a small number of bridge points. If $S$ is connected, one can \emph{meridionally stabilize} the trisection; this decreases the number of bridge points by modifying the trisection of $X$ in a way that increases the trisection genus.

\begin{definition}
Suppose that a closed 4-manifold $X$ has a $(g;k_1,k_2,k_3)$-trisection $\mathcal{T}$, and that a connected surface $S\subset X$ is in $(b;c_1,c_2,c_3)$-bridge position. Suppose that $c_1\geq 2$. Then there is an arc $t\in \tau_{23}$ connecting different components of $\mathcal{D}_{1}$, and we can define a new trisection $\mathcal{T}'$ of $X$ by:
\begin{itemize}
\item $X_1'=X_1\cup \overline{\nu(\tau)}$;
\item $X_2'=X_2\cup \setminus \nu(\tau)$;
\item $X_3'=X_3\cup \setminus \nu(\tau)$,
\end{itemize}
and a new bridge trisection of $S$ with respect to $\mathcal{T}'$ by:
\begin{itemize}
\item $\mathcal{D}_1'=\mathcal{D}_1\cup (S\cap \overline{\nu(\tau)})$;
\item $\mathcal{D}_2'=\mathcal{D}_2\setminus \nu(\tau))$;
\item $\mathcal{D}_3'=\mathcal{D}_3\setminus \nu(\tau))$;
\end{itemize}
The trisection $\mathcal{T}'$ is called a \emph{meridional 1-stabilization} of $\mathcal{T}$. Meridional 2- and 3-stabilization are defined similarly.
\end{definition}

Meier and Zupan show \cite[Lemma 22]{MeiZup18} that $\mathcal{T}'$ is indeed a $(g+1;k_1+1,k_2,k_3)$-trisection for $X$, and that $S$ is in $(b-1;c_1-1,c_2,c_3)$-bridge position with respect to $\mathcal{T}'$. In particular, by repeated meridional stabilization, one can arrange for a connected surface $S\subset X$ to be in $(b;1)$-bridge trisected position with respect to some trisection of $X$, and for some $b\geq 1$. Note that if $S$ is in $(b;c)$-bridge position then $\chi(S)=3c-b$, and so an embedded 2-sphere can always be put in $(1;1)$-bridge position with respect to some trisection.

\begin{example}
To meridionally stabilize the bridge trisection of the spun trefoil $S$ in Figure \ref{ShadowDiagramSpunTrefoil}, note that the arc $\tau$ connects the two components bounded by the union of the red and blue tangles. The stabilization adds the annulus $\partial\nu(\tau)\cap (X_1\cap X_2)$ to the central surface $\Sigma$. In the schematic, the two open circles in each image are identified to form a torus. The surface $S$ now intersects $\Sigma$ in only 6 points, and has the illustrated shadows.

\begin{figure}[ht] 
	\centering
	\includegraphics[width=\linewidth]{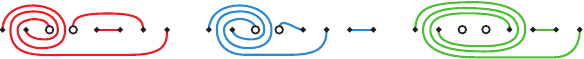}
	\caption{Meridionally stabilizing the bridge trisection in Figure \ref{ShadowDiagramSpunTrefoil} once gives the following schematic, which records a $(3;1,2,2)$-bridge trisection of $S$ with respect to a $(1;1,0,0)$-trisection of $S^4$.}%
	\label{TriplaneDiagramSpunTrefoilStabilized}
\end{figure}
\end{example}

For more exposition on bridge trisections and the various stabilization operations, as well as many more examples, see \cite{MeiZup18}.

\section{Trisection Diagrams of Surface Complements}\label{sec:surfacecomplements}

In this section, we summarize some recent results on relative trisection diagrams of complements of surfaces in 4-manifolds. Suppose that $X$ is a 4-manifold equipped with a trisection $\mathcal{T}$ and that $S\subset X$ is a connected embedded surface. By \cite{MeiZup18}, one can isotope $S$ to be in bridge position with respect to $\mathcal{T}$. In fact, one can meridionally stabilize $\mathcal{T}$ until $S$ is in $(3-\chi(S);1)$-bridge position with respect to some (stabilized) trisection, which we will continue to denote by $\mathcal{T}$. One might hope that a relative trisection of $X\setminus \nu(S)$ could be obtained by simply deleting a regular neighbourhood of $S$ from each sector of $\mathcal{T}$. In fact, if $S$ is a sphere in 1-bridge position, then a relative trisection of $X\setminus\nu(S)$ can be obtained in this way.
If $S$ is not a sphere, then this procedure \emph{never} produces a relative trisection of $X\setminus\nu(S)$. Indeed, relative trisections are required to induce an open book decomposition on $\partial X$ for which $X_i\cap X_j\cap \partial X$ are (connected) pages. In general, if $S$ is in $(b;1)$-bridge position then $(X_i\setminus\nu(S))\cap (X_j\setminus \nu(S))\cap \partial (X\setminus \nu(S))=\sqcup_b S^1\times I$, which is disconnected unless $S$ is in 1-bridge position (and $S$ is a sphere). However, this decomposition of $X\setminus\nu(S)$ can be improved to a trisection using the boundary stabilization technique developed in \cite{KimMil1808}.

\begin{definition}[\hspace{1sp}\cite{KimMil1808}]
Let $X$ be a smooth, oriented, closed, and connected 4-manifold with connected non-empty boundary, and suppose that $X=X_1\cup X_2\cup X_3$ where $\textnormal{int}(X_i)\cap \textnormal{int}(X_j)=\emptyset$. Let $c$ be an arc in $X_i\cap X_j\cap \partial X$, and let $\nu(c)$ be a fixed open tubular neighbourhood of $c$. Define:
	\begin{itemize}
		\item $X_i'=X_i\setminus \nu(c)$;
		\item $X_j'=X_j\setminus \nu(c)$;
		\item $X_k'=X_k\cup \nu(c)$.
	\end{itemize}
	The replacement $(X_1,X_2,X_3)\to (X_1',X_2',X_3')$ is called a \emph{boundary stabilization}.
\end{definition}

Kim and Miller show that in general, a relative trisection for $X\setminus \nu(S)$ can be obtained by deleting a tubular neighbourhood of $S$ from each sector, and then boundary stabilizing the resulting decomposition sufficiently many times (thus connecting the components of $\partial X\cap X_i\cap X_j$). For more details, we refer the reader to \cite{KimMil1808}. As a brief example, we will consider the case of projective planes embedded in $S^4$.

\begin{definition}
Let $M_\pm\subset S^3\subset S^4$ be the standard M\"obius band with either a positive or negative half twist. Pushing the boundary of $M_\pm$ into the upper hemisphere of $S^4$ and capping the resulting unlink with a disk gives an embedded projective plane $P_\pm\subset S^4$, which we will refer to as \emph{unknotted}. These two embeddings are distinguished up to isotopy by their normal Euler numbers, since $e(P_\pm)=\pm2$.
\end{definition}

A triplane diagram of the unknotted projective plane $P_+\subset S^4$ is given in Figure \ref{TriplaneDiagramPMinus}. In fact, $S^4\setminus \nu(P_+)\cong \nu(P_-)$, and so $S^4=\nu(P_+)\cup \nu(P_-)$ \cite{KatSaeTerYam99}. After boundary stabilizations, Kim and Miller obtain the relative trisection diagram for $S^4\setminus \nu(P_-)$ given in Figure \ref{TrisectionDiagramRP2ext}. One can also verify that this diagram is correct using the usual algorithm to extract a Kirby diagram from a trisection diagram. The mirror image of this diagram is a diagram for $\nu(P_-)$.

\begin{figure}[ht]
	\centering
	\includegraphics[width=\linewidth]{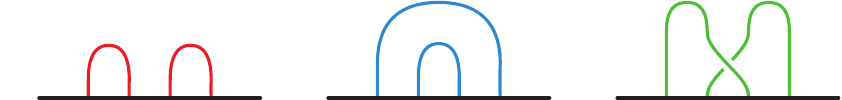}
	\caption{A triplane diagram describing a $(2,1)$-bridge trisection of $P_-$ with respect to the $(0;0)$-trisection of $S^4$.}%
	\label{TriplaneDiagramPMinus}
\end{figure}

\begin{figure}[ht]
	\centering
	\includegraphics[height=0.5\linewidth]{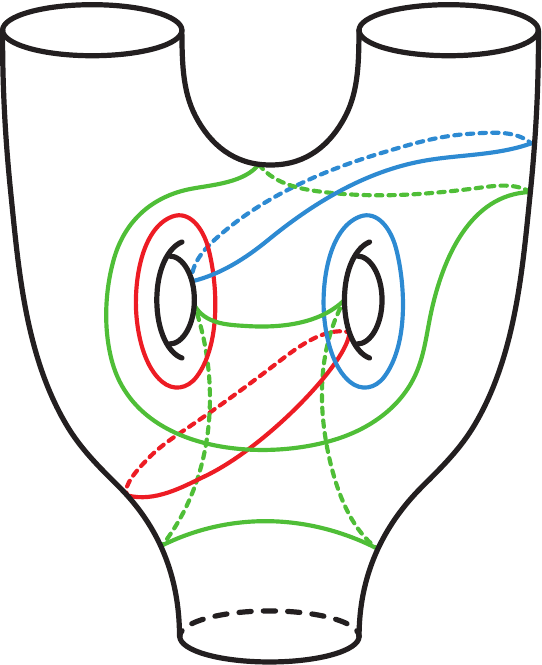}
	\caption{A $(2,2;0,3)$-relative trisection diagram of $S^4\setminus \nu(P_-)$ obtained by starting with the bridge trisection of $P_-$ in Figure \ref{TriplaneDiagramPMinus} above, removing $\nu(P_-)$ from $S^4$, and performing boundary stabilizations.}%
	\label{TrisectionDiagramRP2ext}
\end{figure}

Gay and Meier \cite{GayMei1806} studied the special case of surgery along 2-spheres in detail. Suppose that $S$ is a 2-sphere embedded in a trisected 4-manifold $X$ with trivial normal bundle, and that $S$ is in $(1;1)$-bridge position. Then $X\setminus \nu(S)$ inherits a natural trisection which induces an open book decomposition on $\partial(X\setminus \nu(S))\cong S^1\times S^2$ with annular pages (and identity monodromy). In general, a relative trisection diagram is called \emph{$p$-annular} if the pages of the induced open book are annuli, and the induced monodromy is given by $p$ Dehn twists about the core of the annulus. In particular, the boundary of the described 4-manifold is the lens space $L(p,1)$. Gay and Meier show that if $\mathcal{T}$ and $\mathcal{T}'$ are $p$-annular relative trisections of $X$ and $X'$, then there is a unique way to glue the associated trisections together. Moreover, given relative trisection diagrams $\mathfrak{D}$ and $\mathfrak{D}'$ for $\mathcal{T}$ and $\mathcal{T}'$, the resulting diagram is \emph{independent} of the choices of arcs extending $\mathfrak{D}$ and $\mathfrak{D}'$ to arced relative trisection diagrams.

Consequently, one way to produce a trisection diagram for the Gluck twist $\Sigma_S(X)$ of $X$ along $S$ is to glue a 0-annular relative trisection diagram for $S^2\times D^2$ to a 0-annular relative trisection diagram for $X\setminus \nu(S)$ (via the twist map $\tau:S^2\times S^1\to S^2\times S^1$ from Section \ref{sec:intro}). Theorem C of \cite{GayMei1806}  provides a short cut-and-paste method to produce such a diagram.

\begin{theorem*}[\hspace{1sp}\cite{GayMei1806}]
Let $X$ be a 4-manifold and suppose $S\subset X$ is an embedded 2-sphere. Suppose that $\mathfrak{D}^0$ is an arced trisection diagram for the complement $X\setminus \nu(S)$. Then the result of Gluck surgery along $S$ in $X$ is described by the trisection diagram $\mathfrak{D}^0\cup \mathfrak{D}_\mathfrak{a}$, as in Figure \ref{GayMei1806TheoremC}.
\end{theorem*}

The content of this theorem is illustrated in Figure \ref{GayMei1806TheoremC}. Here, $\mathfrak{D}_\mathfrak{a}$ is the annular diagram consisting of two parallel $b,c$ arcs and an $a$ arc that differs by a positive Dehn twist. It is important to note that $\mathfrak{D}_\mathfrak{a}$ is \emph{not} a relative trisection diagram for $S^2\times D^2$, but features as though it is. It arises as the result of destabilizing the diagram obtained by gluing $\mathfrak{D}^0$ to an honest relative trisection diagram for $S^2\times D^2$. The result is also true if we replace $\mathcal{D}_\mathfrak{a}$ with the analogous diagrams $\mathcal{D}_\mathfrak{b}$ or $\mathcal{D}_\mathfrak{c}$, or their mirrors (Remark 5.6 \cite{GayMei1806}).

\begin{figure}[ht] 
	\centering
	\includegraphics[width=0.7\linewidth]{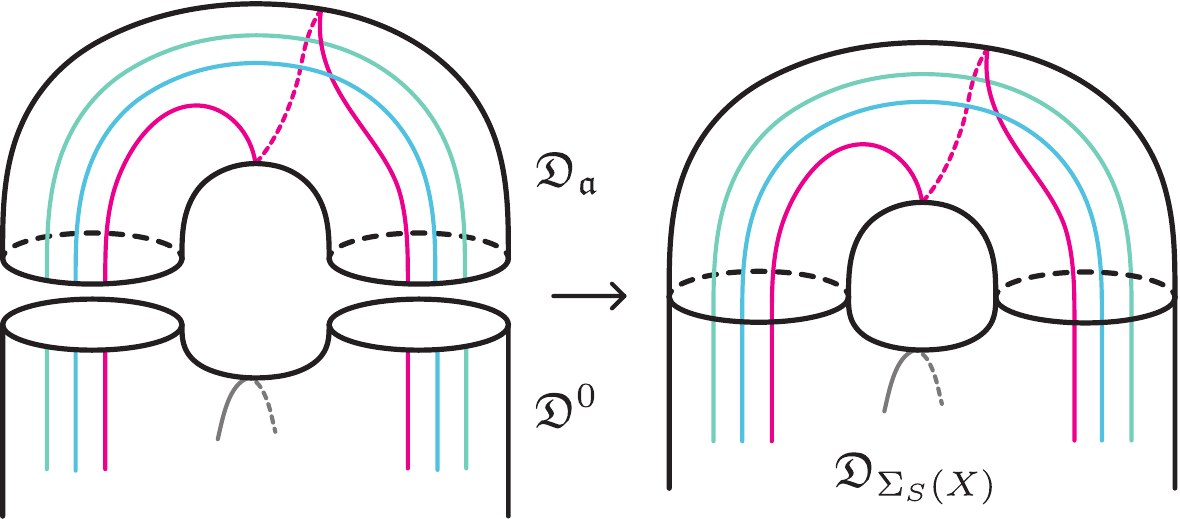}
	\caption{Performing Gluck surgery by gluing diagrams.}%
	\label{GayMei1806TheoremC}
\end{figure}

Here and in the next section, we will draw a grey arc in a diagram (e.g. $\mathfrak{D}^0$ in Figure \ref{GayMei1806TheoremC}) to indicate that this portion of the diagram may contain many curves of arbitrary colors. For clarity, we will also color arcs in a trisection diagram lighter than closed curves.

\section{Diagramatic Proof}\label{sec:proof}

\subsection{Reducing to diagrams}

We will now give a new proof of Theorem \ref{kstythm}. We will begin by precisely formulating a diagrammatic statement that implies the result, and then carry out a proof using these diagrams.

\begin{proposition}
Let $X$ be a smooth, oriented, closed and connected 4-manifold, and let $S\subset X$ be an embedded 2-sphere with trivial normal bundle. Let $P_\pm\subset X$ be an unknotted projective plane. Then the manifolds $\Sigma_S(X)$ and $\Pi_{S\# P_{\pm}}(X)$ are diffeomorphic if the portion of the trisection diagram illustrated in Figure \ref{DiagramStartingPoint} can be converted (through a sequence of handle slides, isotopy of curves, surface diffeomorphisms and destabilizations) to one of $\mathfrak{D}_\mathfrak{a}$,$\mathfrak{D}_\mathfrak{b}$, or $\mathfrak{D}_\mathfrak{c}$.
\end{proposition}

\begin{figure}[ht]
  \centering
  \includegraphics[width=0.95\textwidth]{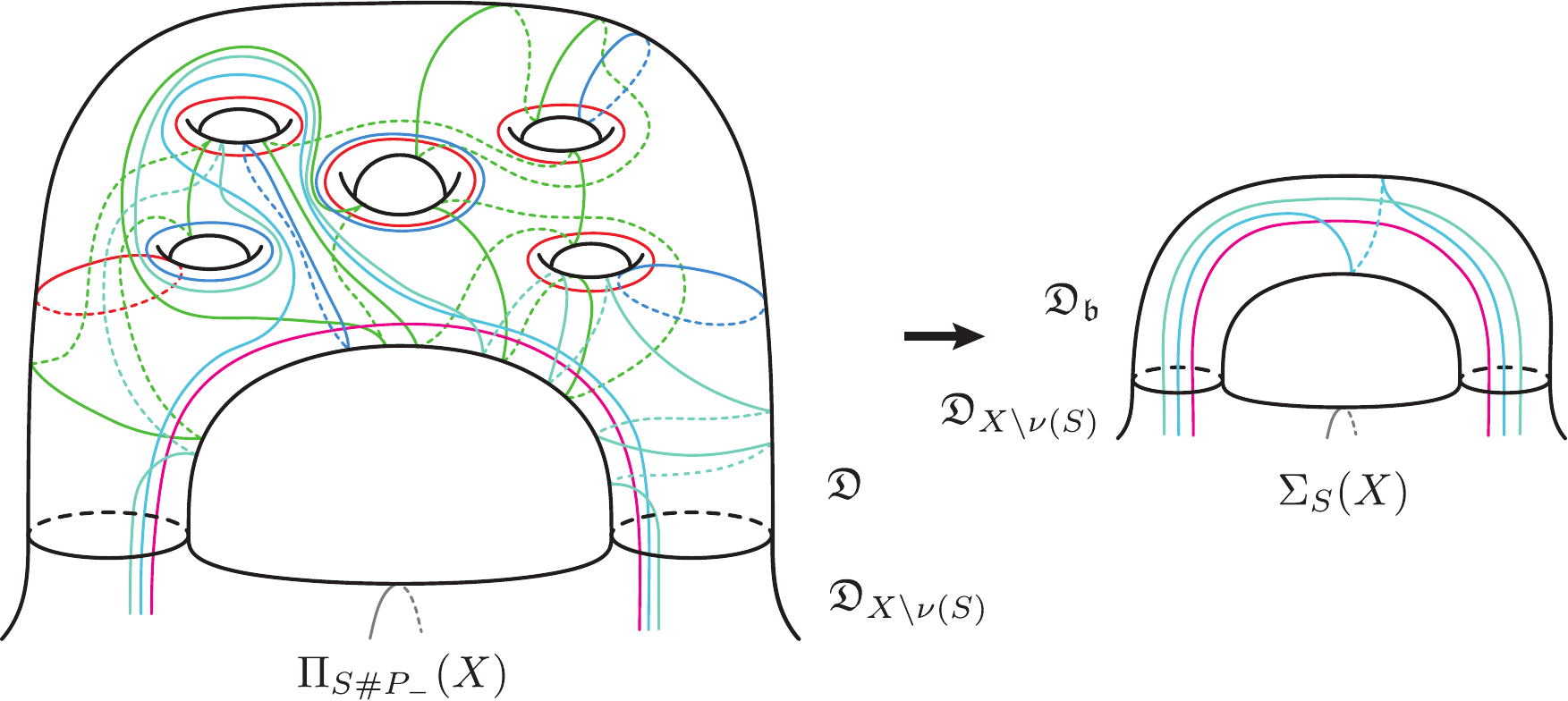}
  \caption{A portion of a trisection diagram for $\Pi_{S\#P_\pm}$.}%
	\label{DiagramStartingPoint}
\end{figure}

\begin{remark}
The large diagram on the left of Figure \ref{DiagramStartingPoint} is only \emph{part} of a larger trisection diagram for a closed 4-manifold. Such a diagram is not necessarily guaranteed to be an honest relative trisection diagram, even though the illustration is suggestive. Similarly, $\mathfrak{D}_\mathfrak{b}$ (on the right of Figure \ref{DiagramStartingPoint}) is not a relative trisection diagram. 
\end{remark}

\begin{proof}
	Let $\mathcal{T}$ be a trisection of $X$. By \cite{MeiZup18}, the 2-sphere $S$ can be isotoped to lie in bridge position with respect to $\mathcal{T}$. Furthermore, by repeated meridional stabilization, $S$ can be assumed to be in $(1;1)$-bridge position with respect to a stabilization of $\mathcal{T}$ (which we will continue to denote by $\mathcal{T}$). Consequently, $X\setminus\nu(S)$ inherits a natural $0$-annular relative trisection, i.e., the induced open book decomposition of $\partial (X\setminus \nu(S))=S^2\times S^1$ has annular pages and trivial monodromy. Let $\mathfrak{D}_{X\setminus \nu(S)}$ be a relative trisection diagram describing this relative trisection of $X\setminus \nu(S)$.
	
	Now, let $P_\pm\subset X$ be an unknotted projective plane with Euler number $\pm 2$. By the gluing results in \cite[Section 5]{KimMil1808}, a relative trisection diagram for $X\setminus \nu(S\#P_\pm)$ can be obtained as $\mathfrak{D}_{X\setminus \nu(S)}\cup \mathfrak{D}_{S^4\setminus \nu(P_\pm)}$. This is illustrated in the schematic in Figure \ref{schematic} below, for the case of $P_-$. For clarity, the arcs in these relative trisection diagrams have been omitted. They appear in full in Figure \ref{DiagramStartingPoint}. 

	We have now constructed a relative trisection diagram for $X\setminus \nu(S\#P_\pm)$, and need to glue in the neighbourhood $N_\pm$ via the Price surgery map. By \cite[Corollary 5.3]{KimMil1808} a trisection diagram for $\Pi_{S\#P_\pm}(X)$ can be obtained by gluing together our diagram for $X\setminus \nu(S\#P_\pm)$ and a relative trisection diagram $\mathfrak{D}_{\nu(P_\pm)}$ for $N_\pm$, as in the schematic. It is important to note that this must be performed carefully; the surgery dictates which boundary components are identified. In fact, by \cite[Lemma 5.1]{KimMil1808} this choice essentially determines the arcs of the diagram, since the monodromy of $Q=\partial N_\pm$ is highly constrained (it consists of two Dehn twists about each boundary component, with signs as indicated). After using the monodromy algorithm of \cite{CasGayPin18a} to complete Figure \ref{schematic} with arcs (in the case of $P_-$), one obtains the trisection diagram for $\Pi_{S\#P_\pm}(X)$ illustrated on the left of Figure \ref{DiagramStartingPoint}.
	
	\begin{figure}[ht]
    	\includegraphics[width=0.7\textwidth]{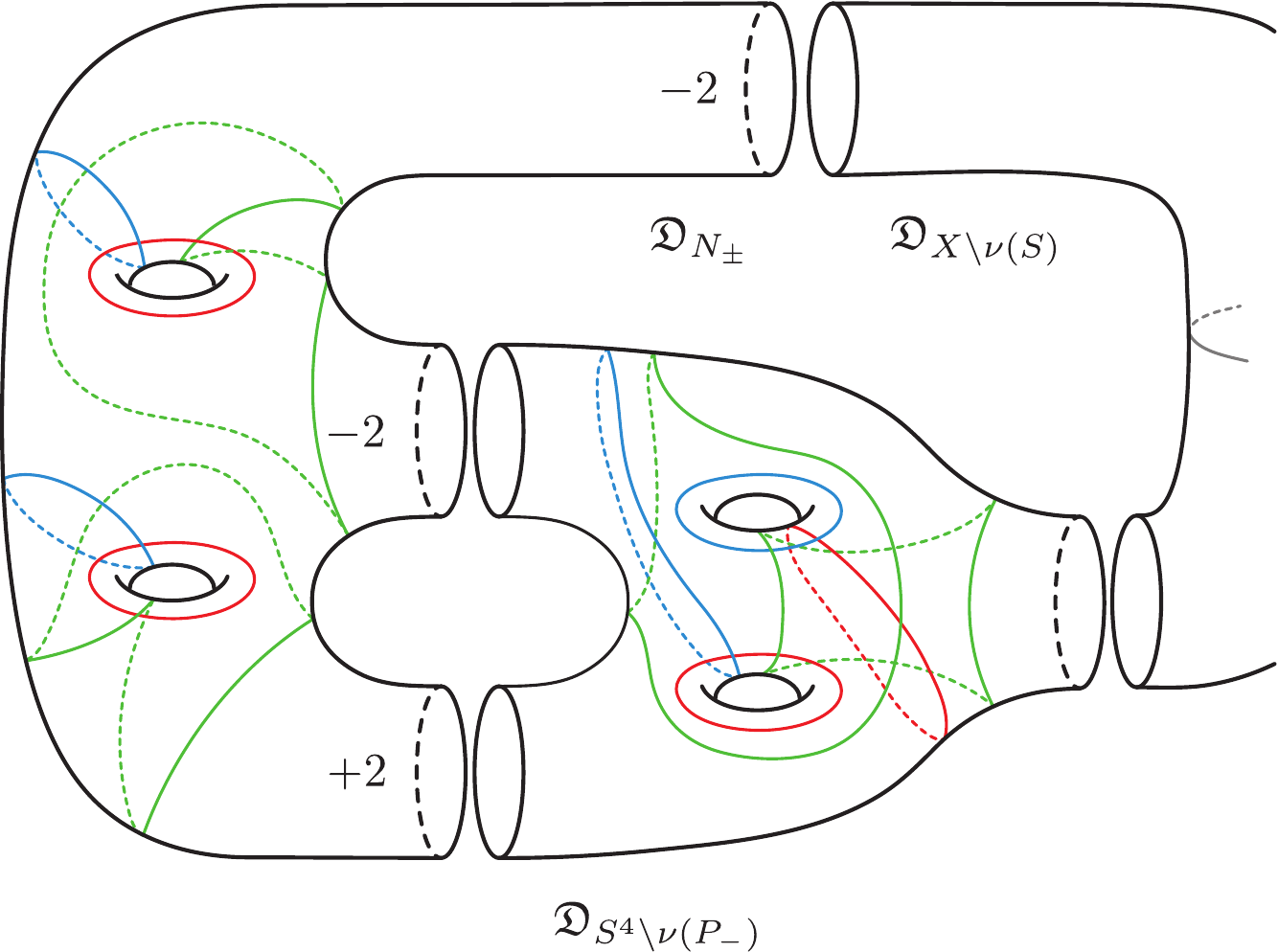}
    \caption{The origin of Figure \ref{DiagramStartingPoint}.}%
		\label{schematic}
    \end{figure}
    
    On the other hand, constructing a trisection diagram for $\Sigma_S(X)$ is more straightforward. By \cite[Theorem C]{GayMei1806}, such a diagram is given by $\mathfrak{D}_{X\setminus \nu(S)}\cup\mathfrak{D}_\mathfrak{b}$, together with arcs for each diagram. This is illustrated on the right of Figure \ref{DiagramStartingPoint}.
    
    Thus, if $\mathfrak{D}_{S^4\setminus \nu(P_\pm)}\cup \mathfrak{D}_{\nu(P_\pm)}$ can be converted through a sequence of trisection moves (i.e., a sequence of destabilizations, isotopy of curves, handle slides, and surface diffeomorphisms that do not modify $\mathfrak{D}_{X\setminus\nu(S)}$) to give $\mathfrak{D}_\mathfrak{a}$ (or $\mathfrak{D}_\mathfrak{b}$ or $\mathfrak{D}_\mathfrak{c}$), then we will have exhibited the fact that $\Pi_{S\#P_-}(X)$ is diffeomorphic to $\Sigma_S(X)$. In fact, a diagram for $\Pi_{S\#P_+}(X)$ can be obtained using the mirror image of $\mathfrak{D}_{S^4\setminus \nu(P_\pm)}\cup \mathfrak{D}_{\nu(P_\pm)}$ in Figure \ref{schematic}, and so to prove the statement for $P_\pm$ it suffices to prove it for $P_-$.
\end{proof}

\begin{remark}
Since $\Sigma_S(X)$ and $\Pi_{S\#P_\pm}(X)$ are indeed diffeomorphic by \cite{KatSaeTerYam99}, any trisection diagrams for $\Sigma_S(X)$ and $\Pi_{S\#P_\pm}(X)$ can be related by handle slides and isotopy, at least after stabilizations. A priori, one might expect that \emph{both} stabilizations and destabilizations are necessary to carry out the proof in this paper, but surprisingly this turns out not to be the case. Indeed, we will see in the next section that only destabilizations are required. 
\end{remark}

\subsection{Diagrams}

We complete the proof of Theorem \ref{kstythm} by proving the following proposition.

\begin{proposition}
	There is a sequence of destabilizations, isotopy of curves, handle slides, and surface diffeomorphisms that convert the trisection diagram $\mathfrak{D}$ in Figure \ref{DiagramStartingPoint} into $\mathfrak{D}_\mathfrak{b}$.
\end{proposition}

\begin{proof}
    The proof will be a step-by-step verification that this is possible. The strategy will be to perform handle slides and isotopy to transform $\mathfrak{D}$ so that Lemma \ref{lem:destab} applies, destabilize the diagram (i.e., surger a particular curve), and repeat. For organization, we will break the proof into steps.  
    
    \vspace{3mm}
    \noindent \textbf{Step 1:} We start by labelling some curves in Figure \ref{DiagramStartingPoint}. This is illustrated in Figure \ref{DiagramStartingPointLabelled}. We continue to adopt the convention that arcs in diagrams are colored lighter than closed curves, even though these are all closed curves in a trisection diagram for $\Pi_{S\#P_-}(X)$. Since we will perform many handle slides and destabilizations, any labels will be specific to each figure and will change during the proof. We will adopt the standard convention that $\alpha$, $\beta$, and $\gamma$ curves are colored red, blue, and green, respectively.

	\begin{figure}[ht]
		\centering
		\includegraphics[width=0.65\textwidth]{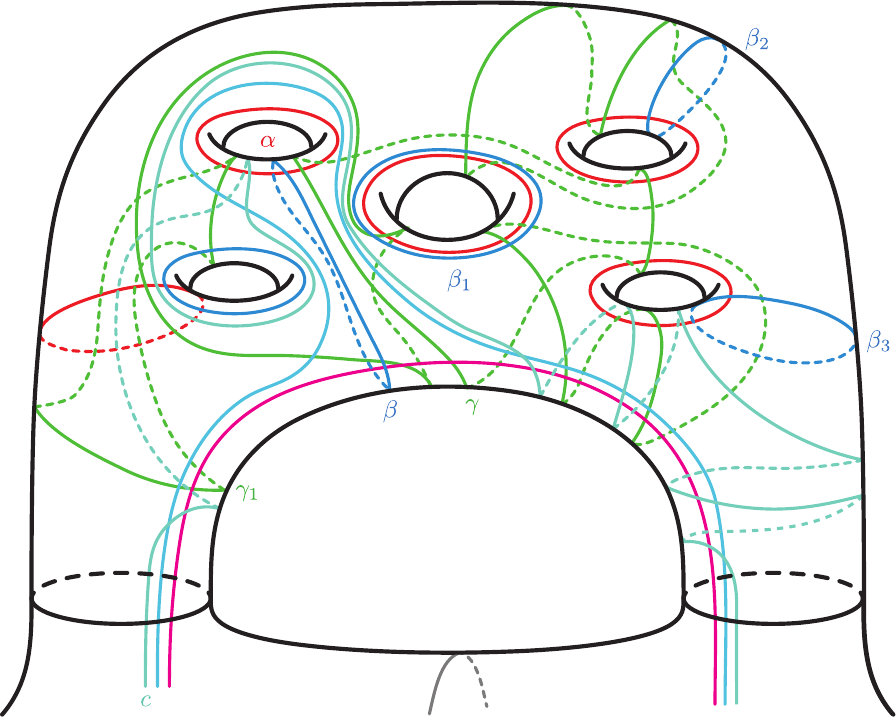}
		\caption{The diagram $\mathfrak{D}$ from Figure \ref{DiagramStartingPoint}, with labels.}%
		\label{DiagramStartingPointLabelled}
	\end{figure}

	Observe that $\alpha$ intersects both $\beta$ and $\gamma$ once. Moreover, we can easily make $\beta$ and $\gamma$ parallel after some handle slides. Specifically, slide $\beta$ over $\beta_1,\beta_2$ and $\beta_3$ to make it parallel to $\gamma$. The result of these slides is illustrated in Figure \ref{Step1}. 
	
		\begin{figure}[ht]
		\centering
		\includegraphics[width=0.65\textwidth]{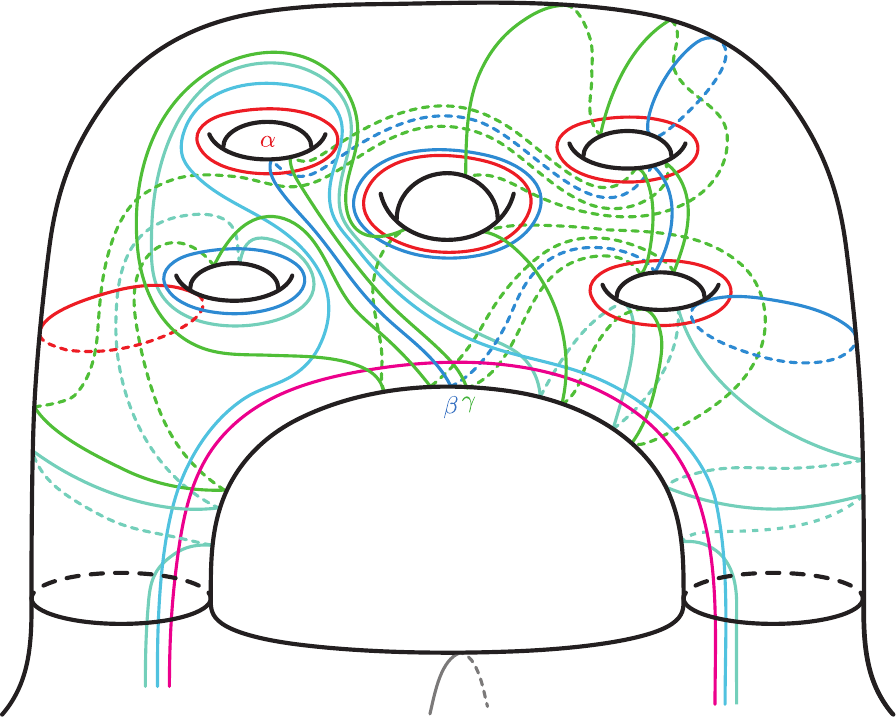}
		\caption{The diagram after performing some handle slides to Figure \ref{DiagramStartingPointLabelled} in Step 1.}%
		\label{Step1}
	\end{figure}

	We can now apply Lemma \ref{lem:destab} to the curves $\alpha$, $\beta$, and $\gamma$ in Figure \ref{Step1} and destabilize this diagram. To destabilize, we surger the surface along the $\alpha$ curve and erase the $\beta,\gamma$ curves. The result of this process is illustrated in Figure \ref{RedrawnAfterStep1}.

	\begin{figure}[ht]
		\centering
		\includegraphics[width=0.65\textwidth]{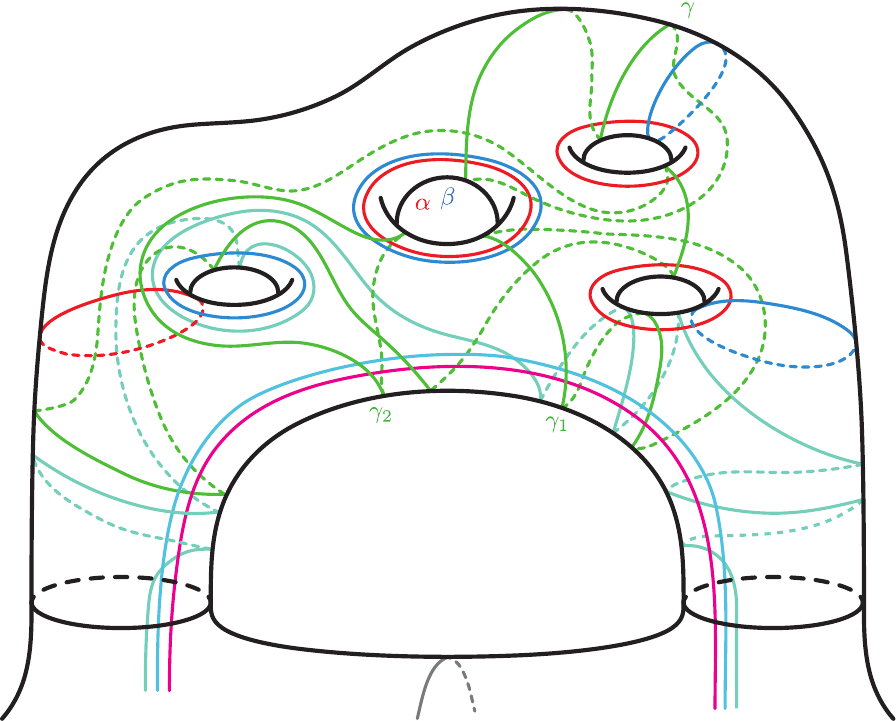}
		\caption{The diagram after destabilizing Figure \ref{Step1} in Step 1.}%
		\label{RedrawnAfterStep1}
	\end{figure}
    
    \vspace{3mm}
    \noindent \textbf{Step 2:}
    We observe that in Figure \ref{RedrawnAfterStep1}, the curves $\alpha$ and $\beta$ are parallel and intersect $\gamma$ once. Consequently, we can slide $\gamma_1$ and $\gamma_2$ over $\gamma$ to obtain the diagram in Figure \ref{Step2}.
    
	\begin{figure}[ht]
		\centering
		\includegraphics[width=0.65\textwidth]{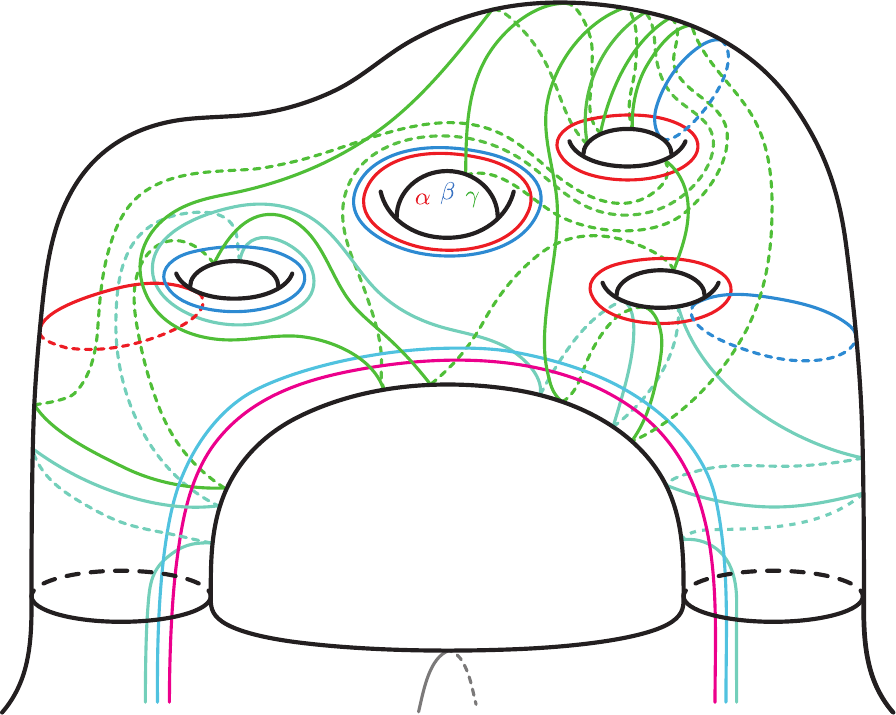}
		\caption{The diagram after performing some handle slides to Figure \ref{RedrawnAfterStep1} in Step 2.}%
		\label{Step2}
	\end{figure}    
    
    In Figure \ref{Step2}, we can now apply Lemma \ref{lem:destab} to the curves $\alpha$, $\beta$, and $\gamma$. To destabilize, we erase the $\beta$ and $\gamma$ curves, and surger the surface along the $\alpha$ curve. The result of this process is illustrated in Figure \ref{RedrawnAfterStep2}.

	\begin{figure}[ht]
		\centering
		\includegraphics[width=0.65\textwidth]{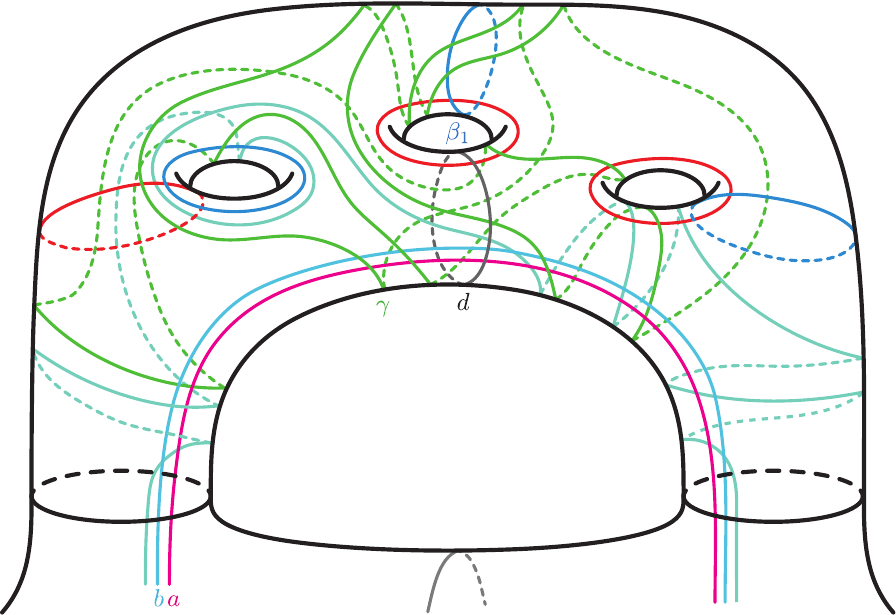}
		\caption{The diagram after destabilizing Figure \ref{Step2} in Step 2.}%
		\label{RedrawnAfterStep2}
	\end{figure}
	
    \vspace{3mm}
    \noindent \textbf{Step 3:}
	We now note that in Figure \ref{RedrawnAfterStep2}, the curve $\gamma$ meets the arcs $a$ and $b$ exactly once. Moreover, since the trisection for $X\setminus \nu(S)$ is 0-annular, $a$ and $b$ can be assumed to be parallel outside of this part of the diagram. Thus after some handle slides, we will be able to destabilize using $a,b$ and $\gamma$.

	In order to do this, we first arrange $\gamma$ to look less complicated. We perform two Dehn twists along the curve $\beta_1$ and one Dehn twist along the curve labelled $d$. After doing this, we obtain the diagram illustrated in Figure \ref{Step3Halfway}.

	\begin{figure}[ht]
		\centering
		\includegraphics[width=0.65\textwidth]{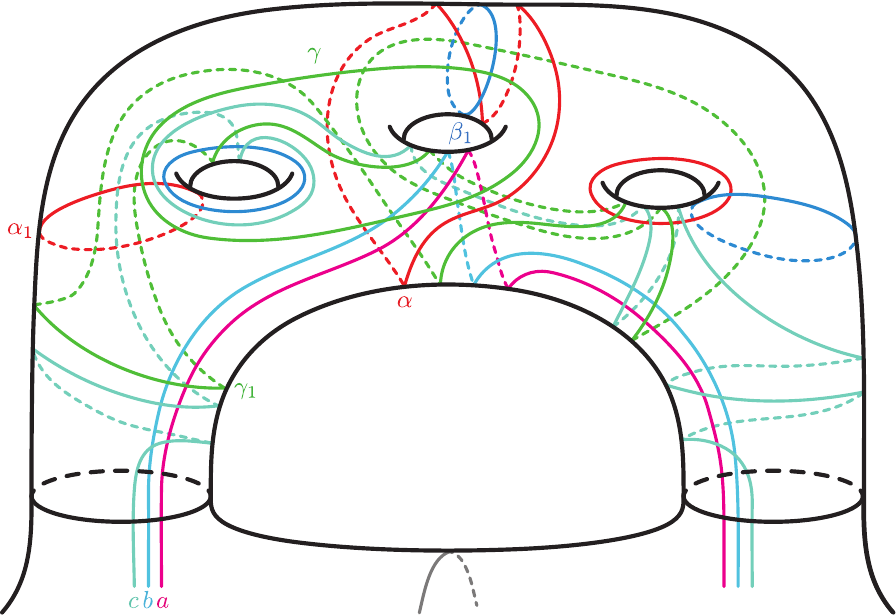}
		\caption{The diagram in Figure \ref{RedrawnAfterStep2} after three Dehn twists in Step 3.}%
		\label{Step3Halfway}
	\end{figure}

	Now that $\gamma$ looks simpler, we perform some additional handle slides so that we may appeal to Lemma \ref{lem:destab}. In Figure \ref{Step3Halfway}, slide $\alpha$ over $\alpha_1$, and then $\alpha_1$ over $a$. Next, slide $\beta_1$ over $b$. Last, slide $c$ over $\gamma_1$ and $\gamma$. Note that although $a$, $b$, and $c$ appear as arcs, they are actually closed curves in this trisection diagram. This process removes all extraneous intersections with $\gamma$, and the result of these handle slides is illustrated in Figure \ref{Step3}.

	\begin{figure}[ht]
		\centering
		\includegraphics[width=0.65\textwidth]{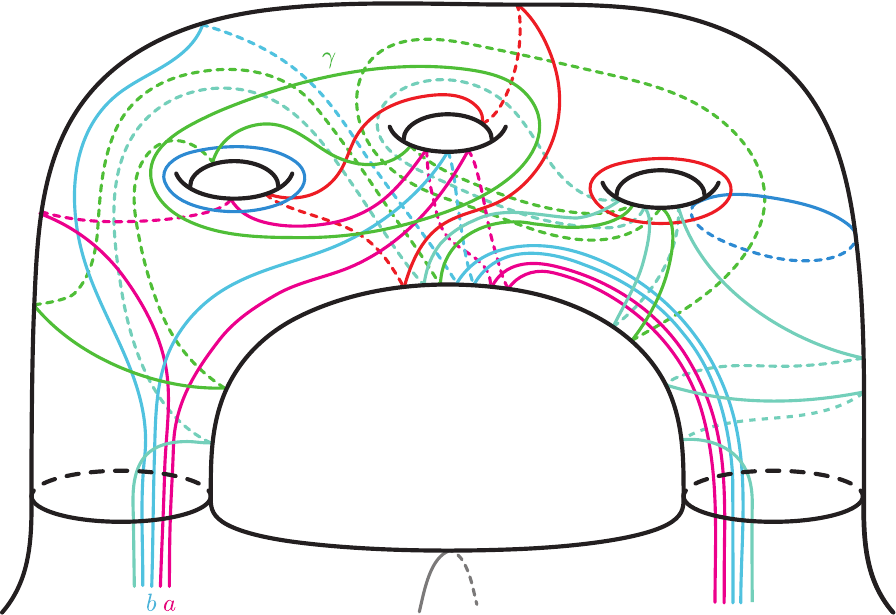}
		\caption{The diagram in Figure \ref{Step3Halfway} after performing some handle slides in Step 3.}%
		\label{Step3}
	\end{figure}

	We now use Lemma \ref{lem:destab} to destabilize the diagram in Figure \ref{Step3} using the curves $a$, $b$, and $\gamma$. To do so, we erase $a$ and $b$ and surger the surface using $\gamma$. This takes slightly more visualizing than the previous two destabilizations, but the result after a mild isotopy is illustrated in Figure \ref{RedrawnAfterStep3}.
	
	Note that while this process removes the curves $a$ and $b$ from $\mathfrak{D}_{X\setminus \nu(S)}$, our earlier slides produced a second copy of these curves, and so $\mathfrak{D}_{X\setminus \nu(S)}$ remains unchanged. 
	
	\begin{figure}[ht]
		\centering
		\includegraphics[width=0.65\textwidth]{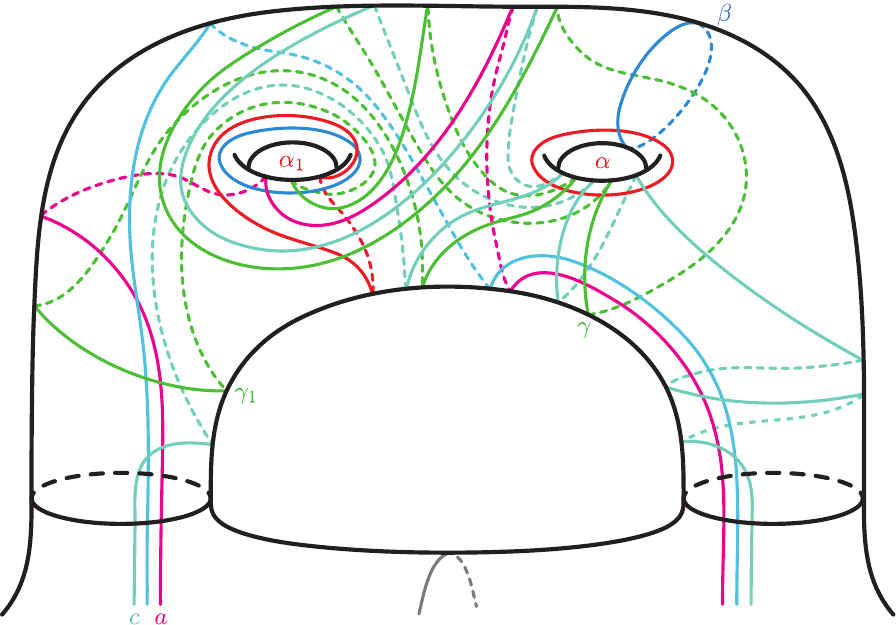}
		\caption{The diagram after destabilizing Figure \ref{Step3} in Step 3.}%
		\label{RedrawnAfterStep3}
	\end{figure}
	
	\vspace{3mm}
	\noindent \textbf{Step 4:}
	This step is similar to Step 3, but slightly more involved. We note that in Figure \ref{RedrawnAfterStep3}, the curve $\beta$ intersects the $\alpha$ and $\gamma$ curves each once. If we can arrange $\alpha$ and $\gamma$ to be parallel, we will be able to use Lemma \ref{lem:destab} to destabilize the diagram again.

	To this end, perform a Dehn twist to make the curve $\alpha_1$ appear as a standard longitude of the leftmost hole of the surface. Then, to simplify the diagram, slide the curve $c$ over $\gamma$. Now slide $a$ over both $\alpha$ curves repeatedly so that it is parallel to $c$. Lastly, slide $\gamma$ over $c$ and $\alpha$ over $a$ so that these curves no longer intersect $\beta$.
	
	\begin{figure}[ht]
		\centering
		\includegraphics[width=0.65\textwidth]{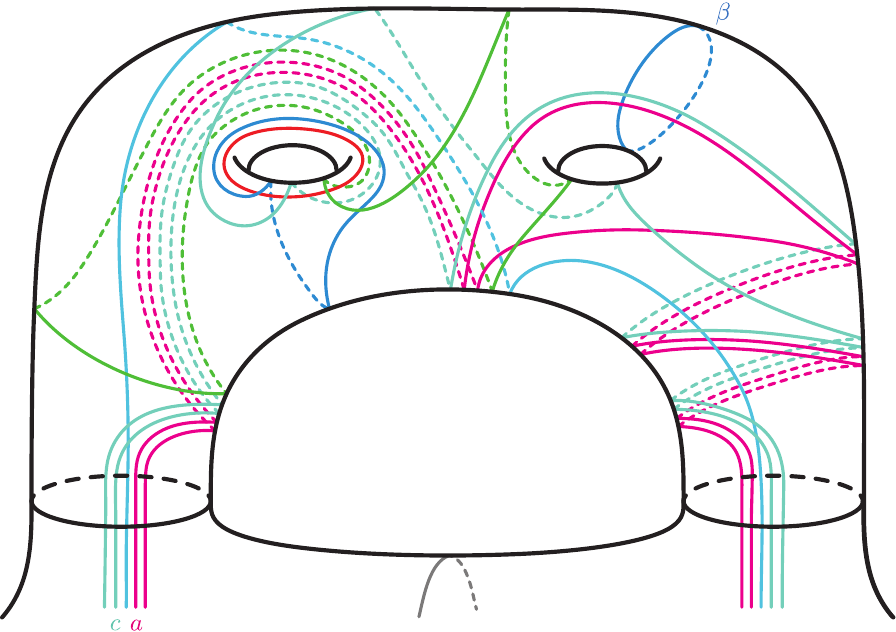}
		\caption{The diagram in Figure \ref{RedrawnAfterStep3} after one Dehn twist and several handle slides in Step 4.}%
		\label{Step4}
	\end{figure}

	The result after this Dehn twist and these handle slides is illustrated in Figure \ref{Step4}. Using Lemma \ref{lem:destab}, we can now use the curves $a$, $\beta$, and $c$ to destabilize the diagram. To do so, we erase the $a$ and $c$ curves, and surger the surface using $\beta$. The result of this process is illustrated in Figure \ref{RedrawnAfterStep4}. As before, $\mathfrak{D}_{X\setminus \nu(S)}$ remains unchanged. 

	\begin{figure}[ht]
		\centering
		\includegraphics[width=0.65\textwidth]{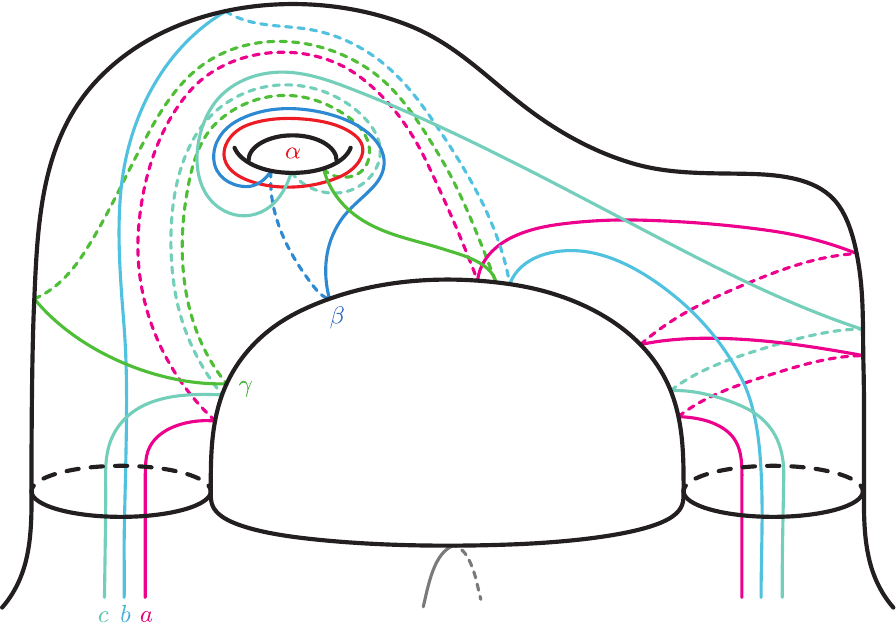}
		\caption{The diagram in Figure \ref{Step4} after a destabilization in Step 4.}%
		\label{RedrawnAfterStep4}
	\end{figure}
	
	\vspace{3mm}
	\noindent \textbf{Step 5:}
	We only need to perform one more destabilization. In Figure \ref{RedrawnAfterStep4}, slide $a$ over $\alpha$ twice. Next, do a Dehn twist along $\alpha$ to make the curve $\beta$ appear as a standard meridian of the hole in the surface. We can now slide $b$ over $\beta$ to make it parallel to $c$, and the resulting pair of curves both intersect $\alpha$ exactly once. 
	To apply Lemma \ref{lem:destab}, we only need to perform handle slides to remove all other intersections with $\alpha$. To do this, slide $\beta$ over the new $b$ curve and $\gamma$ over $c$. The result of this Dehn twist and these handle slides is illustrated in Figure \ref{Step5}.
	
	\begin{figure}[ht]
		\centering
		\includegraphics[width=0.65\textwidth]{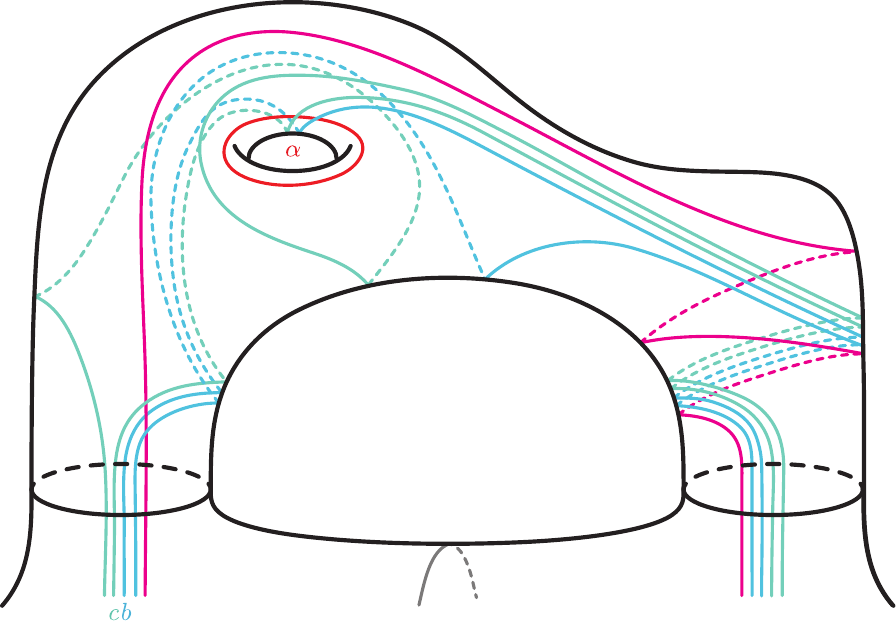}
		\caption{The diagram in Figure \ref{RedrawnAfterStep4} after performing a Dehn twist and several handle slides in Step 5.}%
		\label{Step5}
	\end{figure}
	
	We can now apply Lemma \ref{lem:destab}, and destabilize the diagram in Figure \ref{Step5} using the curves $\alpha$, $b$, and $c$. To do this, we erase the $b$ and $c$ and surger the surface using $\alpha$. The result of this process is illustrated in Figure \ref{RedrawnAfterStep5}.

	\begin{figure}[ht]
		\centering
		\includegraphics[width=0.5\textwidth]{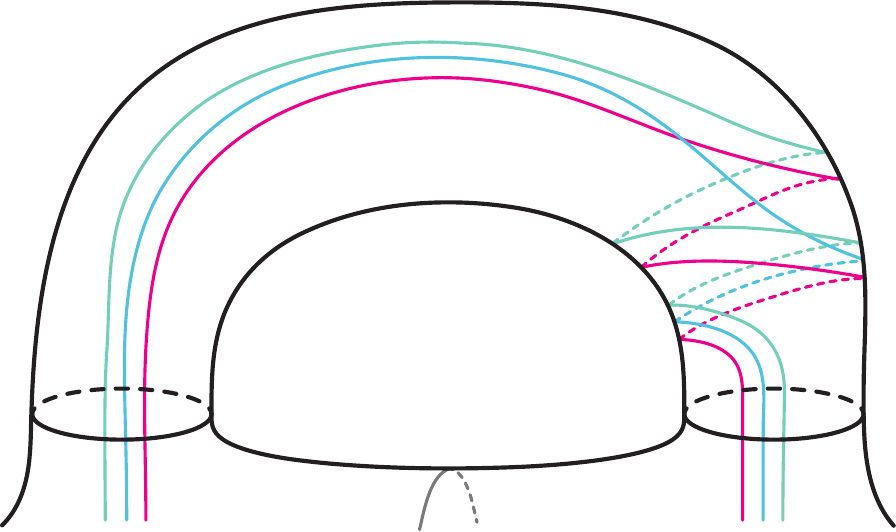}
		\caption{The diagram in Figure \ref{Step5} after a destabilization in Step 5. After Dehn twists, this diagram is equivalent to the diagram on right side of Figure \ref{DiagramStartingPoint}.}%
		\label{RedrawnAfterStep5}
	\end{figure}

	Up to Dehn twists, we see that the diagram in Figure \ref{RedrawnAfterStep5} is in fact equivalent to the diagram $\mathfrak{D}_{\mathfrak{b}}$.	This completes the proof.
\end{proof}

\section{Further questions}\label{sec:questions}

Even though the proof in Section \ref{sec:proof} is a seemingly ad hoc sequence of moves, one might hope to apply similar trisection diagrammatic methods to show that various homotopy 4-spheres are standard. Unlike Kirby diagrams, trisection diagrams offer three seemingly symmetric possible destabilizations. It would be interesting to see if this additional flexibility provides any insight into the handle decompositions of any homotopy 4-spheres that are not known to be diffeomorphic to $S^4$. 

In particular, \cite[Theorem C]{GayMei1806} gives a potential method to show that a given Gluck twist $\Sigma_S(S^4)$ is standard. Starting with an embedded 2-sphere $S$ in $(1;1)$-bridge position, one could attempt to destabilize the resulting trisection diagram of $\Sigma_S(S^4)$ to one which describes $S^4$.  

\begin{question}
Can trisection diagrammatic methods be used to understand the handle structure of homotopy 4-spheres such as Gluck twists?
\end{question}

However, even if $\Sigma_S(S^4)\cong S^4$, it remains an open question whether all trisection diagrams for $S^4$ are \emph{standard}, i.e., are stabilizations of the $(0;0)$-trisection of $S^4$. Whether this is the case is a conjecture of Meier-Schirmer-Zupan \cite{MeiSchZup16}.

\begin{conjecture}
Every trisection of $S^4$ is isotopic to the $(0;0)$-trisection or a stabilization of the $(0;0)$-trisection. 
\end{conjecture}

\bibliography{References}
\bibliographystyle{alpha}

\end{document}